\documentclass[reqno]{amsart}
\usepackage[utf8]{inputenc}
\def\bigskip{\vspace{14pt}}

\usepackage[utf8]{inputenc} %test
\usepackage{amssymb}
\usepackage{mathrsfs}
\usepackage{graphicx}
\usepackage{latexsym}
\usepackage{stmaryrd}
\usepackage{enumitem}
\usepackage[bookmarks,hyperfootnotes=false]{hyperref}
\usepackage{multirow}
\usepackage{xcolor}
\usepackage{caption}
\colorlet{darkishRed}{red!80!black}
\colorlet{darkishBlue}{blue!60!black}
\colorlet{darkishGreen}{green!60!black}
\hypersetup{
    draft = false,
    bookmarksopen=true,
    colorlinks,
    linkcolor={red!60!black},
    citecolor={green!60!black},
    urlcolor={blue!60!black}
}
\usepackage[nameinlink, capitalise, noabbrev]{cleveref}
\crefformat{enumi}{#2#1#3}
\crefformat{equation}{#2(#1)#3}
\usepackage[msc-links]{amsrefs} 
\usepackage{doi}

\renewcommand{\PrintDOI}[1]{\doi{#1}}

\let\setminus=\smallsetminus
\usepackage{tikz}
\usepackage{tikz-cd}
\usetikzlibrary{calc,through,intersections,arrows, trees, positioning, decorations.pathmorphing, cd}
\usepackage{relsize}
\usepackage{comment}
\usepackage{svg}
\usepackage{mathtools}
\usepackage{nccmath}
\usepackage{pifont}
\usepackage{comment}

\usepackage[skip=6pt]{subcaption} % mehrere Abbildungen nebeneinander/übereinander
\usepackage{pgfplots}
\usepgfplotslibrary{fillbetween}
\usetikzlibrary{patterns}
\pgfdeclarelayer{ft}
\pgfdeclarelayer{bg}
\pgfsetlayers{bg,main,ft}
\pgfplotsset{compat=1.18}

\usepackage{geometry}
\let\setminus=\smallsetminus
\renewcommand{\leq}{\leqslant}

\def\namedlabel#1#2{\begingroup
   \def\@currentlabel{#2}%
   \label{#1}\endgroup
}

\let\rho=\varrho
\let\phi=\varphi

\DeclareMathOperator{\interior}{int}
\newcommand{ \N } { \mathbb{N} }

\newcommand{\abs}[1]{\lvert#1\rvert}

\makeatletter

\def\calCommandfactory#1{%
   \expandafter\def\csname c#1\endcsname{\mathcal{#1}}}
\def\frakCommandfactory#1{%
   \expandafter\def\csname frak#1\endcsname{\mathfrak{#1}}}

\newcounter{ctr}
\loop
  \stepcounter{ctr}
  \edef\X{\@Alph\c@ctr}
  \expandafter\calCommandfactory\X
  \expandafter\frakCommandfactory\X
  \edef\Y{\@alph\c@ctr}
  \expandafter\frakCommandfactory\Y
\ifnum\thectr<26
\repeat

\setenumerate{label={\normalfont (\roman*)}}%,itemsep=0pt}

\usepackage[presets={vec-cev,abc,ABC,cAcBcC}]{letterswitharrows}

\newtheorem{theorem}{Theorem}[section] 
\newtheorem{mainresult}{Theorem} 
\crefname{mainresult}{Theorem}{Theorems}

\newtheorem{LEM}[theorem]{Lemma}

\theoremstyle{definition}

\crefname{claim}{Claim}{Claims}

\usepackage{etoolbox}

\newlist{thmlist}{enumerate}{1}
\setlist[thmlist]{label=(\roman{thmlisti}), ref=\thethm.(\roman{thmlisti}),noitemsep}

\def\td{tree-decom\-po\-sition}
\def\tds{tree-decom\-po\-sitions}

\theoremstyle{definition}

\theoremstyle{remark}

\title[Refining tree-decompositions so that they display the \MakeLowercase{k}-blocks]{Refining tree-decompositions so that they display the \MakeLowercase{k}-blocks} 

\author{Sandra Albrechtsen}
\address{Universität Hamburg, Department of Mathematics, Bundesstraße 55 (Geomatikum), 20146 Hamburg, Germany}
\email{sandra.albrechtsen@uni-hamburg.de}

\subjclass{05C83, 05C40, 05C05, 05C63}
\keywords{Tree-decomposition, $k$-block, profile, canonical, locally finite graphs.}

\begin{document}

\maketitle

\begin{abstract}
    Carmesin and Gollin proved that every finite graph has a canonical \td\ $(T, \cV)$ of adhesion less than $k$ that efficiently distinguishes every two distinct $k$-profiles, and which has the further property that every separable $k$-block is equal to the unique part of $(T, \cV)$ in which it is contained. 
    
    We give a shorter proof of this result by showing that such a \td\ can in fact be obtained from any canonical tight \td\ of adhesion less than $k$. For this, we decompose the parts of such a \td\ by further \tds. As an application, we also obtain a generalization of Carmesin and Gollin's result to locally finite graphs.
\end{abstract}

\section{Introduction}

A \emph{$k$-block} in a graph $G$, for some $k \in \N$, is a maximal set of at least $k$ vertices no two of which can be separated in~$G$ by removing fewer than $k$ other vertices. 
For large $k$, the $k$-blocks of a graph are examples of highly connected substructures.

Another example of such a substructure, one which also indicates high local connectivity but is of a more fuzzy kind than blocks, is that of a tangle. Tangles were introduced by Robertson and Seymour in \cite{GMX}.
Formally, a \emph{$k$-tangle} in a graph $G$ is a consistent orientation of all the separations $\{A,B\}$ of $G$ of order less than $k$, as $(A,B)$ say, such that no three such oriented separations together cover the whole graph by the subgraphs induced on their `small sides' $A$. 

Since $k$-blocks cannot be separated by deleting fewer than $k$ vertices, they induce an orientation of every separation of order less than~$k$: towards that side which contains the $k$-block. Although these orientations are consistent in that they all point towards the same $k$-block, they need not be tangles if the $k$-block is too small. But they are \emph{$k$-profiles}: a common generalization of tangles and blocks, in that every $k$-tangle is a $k$-profile, and every $k$-block induces a $k$-profile in the way described above.

Robertson and Seymour \cite{GMX} proved that every finite graph has a \td\ of adhesion less than $k$ that distinguishes all its $k$-tangles, in that they `live' in distinct parts of that \td. Carmesin, Diestel, Hamann and Hundertmark \cite{CDHH13CanonicalAlg} generalized this result by showing that every finite graph has a \td\ of adhesion less than $k$ that distinguishes all its regular $k$-profiles. In addition, the \td\ they constructed has the additional property that it is \emph{canonical}: it is invariant under the automorphisms of the graph.

Carmesin and Gollin improved this result even further and showed that every finite graph $G$ admits a \td\ $(T, \cV)$ as above which additionally displays the structure of the $k$-blocks in $G$, in that every $k$-block in $G$ which can be isolated by any \td\ at all\footnote{We call such $k$-blocks \emph{separable}.} appears as a bag of $(T, \cV)$:

\begin{mainresult}{\cite{CG14:isolatingblocks}*{Theorem 1}}\label{thm:DisplayingkBlocksFinitek}
    Every finite graph $G$ has a canonical \td\ $(T, \cV)$ of adhesion less than $k$ that efficiently distinguishes every two distinct regular $k$-profiles, and which has the further property that every separable $k$-block is equal to the unique bag of $(T, \cV)$ that contains it.
\end{mainresult}

They also proved the following related result:

\begin{mainresult}{\cite{CG14:isolatingblocks}*{Theorem 2}}\label{thm:DisplayingkBlocksFiniteMax}
    Every finite graph $G$ has a canonical \td\ $(T, \cV)$ that efficiently distinguishes every two distinct maximal robust profiles, and which has the further property that every separable block inducing a maximal robust profile is equal to the unique bag of $(T, \cV)$ that contains it.
\end{mainresult}

\noindent See \cref{sec:prelims} for a definition of `robust'.
\medskip

In this paper we give a short proof of \cref{thm:DisplayingkBlocksFinitek,thm:DisplayingkBlocksFiniteMax} by showing the following more general result, which allows us to decompose the parts of a \emph{given} \td\ further, so that the resulting \td\ displays the structure of the blocks:

\begin{mainresult}\label{thm:DisplayingBlocksRefining}
    Let $G$ be any graph, and let $\cB$ a set of separable blocks in $G$. Suppose that $G$ has a tight \td\ $(\tilde{T}, \tilde{\cV})$ that distinguishes all the blocks in $\cB$. Then there exists a \td\ $(T, \cV)$ that refines~$(\tilde{T}, \tilde{\cV})$ and is such that every block in $\cB$ is equal to the unique bag of $(T, \cV)$ that contains it. Moreover, $(T, \cV)$ is canonical if $\cB$ and $(\tilde{T}, \tilde{\cV})$ are canonical.
\end{mainresult}

For the proof of \cref{thm:DisplayingkBlocksFinitek} Carmesin and Gollin gave one particular algorithm to construct a canonical \td\ which distinguishes all $k$-profiles efficiently and which displays all separable $k$-blocks.
However, there are a number of different algorithms to construct canonical \tds\ that distinguish all the $k$-profiles in a graph \cites{CDHH13CanonicalAlg, carmesin2022entanglements, FiniteSplinters}. By \cref{thm:DisplayingBlocksRefining}, we can now choose whichever algorithm we~like\penalty-200\ to~construct an initial \td, perhaps in order to have some control over the structure of those parts that do not contain any blocks, and we can still conclude that the \td\ extends to one which additionally displays all separable $k$-blocks.

Moreover, \cref{thm:DisplayingBlocksRefining} also applies to infinite graphs. Carmesin, Hamann and Miraftab \cite{CanonicalTreesofTDs} and Elbracht, Kneip and Teegen \cite{InfiniteSplinters} showed that every locally finite graph has a canonical \td\ that distinguishes all its $k$-profiles. Moreover, Jacobs and Knappe \cite{PRToTsInLocFinGraphs} showed that every locally finite graph without half-grid minor has a canonical \td\ that distinguishes all its maximal robust profiles. Applying \cref{thm:DisplayingBlocksRefining} to these \tds\ yields the following generalizations of \cref{thm:DisplayingkBlocksFinitek,thm:DisplayingkBlocksFiniteMax}:

\begin{mainresult}\label{thm:DisplayingkBlocksLocFink}
    Every locally finite graph $G$ has a canonical \td\ $(T, \cV)$ of adhesion less than~$k$ that efficiently distinguishes every two distinct $k$-profiles, and which has the further property that every separable $k$-block is equal to the unique bag of $(T, \cV)$ that contains it.
\end{mainresult}

\begin{mainresult}\label{thm:DisplayingkBlocksLocFinMax}
    Every locally finite graph $G$ without half-grid minor has a canonical \td\ $(T, \cV)$ that efficiently distinguishes every two distinct maximal robust profiles, and which has the further property that every separable block inducing a maximal robust profile is equal to the unique bag of $(T, \cV)$ that contains it.
\end{mainresult}

\section{Preliminaries}\label{sec:prelims}

We mainly follow the notions from \cite{DiestelBook16noEE}. In what follows, we recap some definitions which we need later.
All graphs in this paper may be infinite unless otherwise stated. Recall that a graph is \emph{locally finite} if all its vertices have finite degree.

\subsection{Separations}

Let $G$ be any graph. A \emph{separation} of $G$ is a set $\{A, B\}$ of subsets of $V(G)$ such that $A \cup B = V(G)$ and there is no edge in $G$ between $A\setminus B$ and $B \setminus A$. A separation $\{A,B\}$ of $G$ is \emph{proper} if neither $A$ nor $B$ equals $V(G)$. Moreover, $\{A,B\}$ is \emph{tight} if both $G[A\setminus B]$ and $G[B\setminus A]$ contain a component of $G - (A \cap B)$ whose neighbourhood in $G$ equals $A \cap B$. 

The \emph{orientations} of a separation~$\{A, B\}$ are the \emph{oriented separations} $(A,B)$ and $(B,A)$. We refer to~$A$ as the \emph{small side} of $(A,B)$ and to $B$ as the \emph{big side} of $(A,B)$.
If the context is clear, we will simply refer to both oriented and unoriented separations as `separations'.

The \emph{order} of a separation $\{A, B\}$ is the size $\abs{A \cap B}$ of its \emph{separator} $A \cap B$.
For some $k \in \N \cup \{\aleph_0\}$, we define $S_k(G)$ to be the set of all separations of $G$ of order $< k$ and $\vS_k(G) := \{(A,B) : \{A,B\} \in S_k(G)\}$ to be the set of all their orientations.

The oriented separations of a graph $G$ are partially ordered by
$(A, B) \leq (C,D)$ if $A \subseteq C$ and $B \supseteq D$.
A set $\sigma \subseteq \vS_{\aleph_0}(G) \setminus \{(V(G),V(G))\}$ of separations is called a \emph{star} if for any $(A,B), (C,D) \in \sigma$ it holds that $(A,B) \leq (D,C)$. 
The \emph{interior} of a star $\sigma \subseteq \vS_{\aleph_0}(G)$ is the intersection $\interior(\sigma) := \bigcap_{(A, B) \in \sigma} B$. 

The partial order on $\vS_{\aleph_0}(G)$ also relates the proper stars in $\vS_{\aleph_0}(G)$: 
if $\sigma, \tau \subseteq \vS_{\aleph_0}(G)$ are stars of proper separations, then $\sigma \leq \tau$ if and only if for every $(A,B) \in \sigma$ there exists some $(C,D) \in \tau$ such that $(A,B) \leq (C,D)$. Note that this relation is again a partial order \cite{TreeSets}.

\subsection{Profiles}

An \emph{orientation} of $S_k(G)$ is a set $O \subseteq \vS_k(G)$ which contains, for every $\{A,B\} \in S_k(G)$, exactly one of its orientations $(A,B)$ and $(B,A)$.
A subset~$O \subseteq \vS_k(G)$ is \emph{consistent} if it does not contain both $(B,A)$ and $(C,D)$ whenever $(A,B) < (C,D)$ for distinct $\{A,B\}, \{C,D\} \in S_k(G)$.

For some $k \in \N$, we call a consistent orientation $P$ of $S_k(G)$ a \emph{$k$-profile in $G$} if it satisfies that
\[
\text{for all $(A,B), (C,D) \in P$ the separation $(B \cap D, A \cup C)$ does not lie in $P$.} 
\]

A profile is \emph{regular} if it does not contain $(V(G), A)$ for any subset $A \subseteq V(G)$.
For some $n \in \N$, a profile in $G$ is \emph{$n$-robust} if for every $(A, B) \in P$ and every $\{C,D\} \in S_n(G)$ the following holds: if both $(A \cup C, B \cap D)$ and $(A \cup D, B \cap C)$ have order less than $\abs{A \cap B}$, then one of them is contained in $P$. Clearly, every $k$-profile is $k$-robust. A profile is \emph{robust} if it is $n$-robust for every $n \in \N$.

A separation $\{A,B\}$ of $G$ \emph{distinguishes} two profiles in $G$ if they orient~$\{A,B\}$ differently. $\{A,B\}$ distinguishes them \emph{efficiently} if they are not distinguished by any separation of~$G$ of smaller order. 

\begin{LEM}{\cite{InfiniteSplinters}*{Lemma 6.1}}\label{lem:EffDistinguisherAreTight}
Let $P, P'$ be two distinct regular profiles in an arbitrary graph $G$. If $\{A, B\}$ is a separation of finite order that efficiently distinguishes $P$ and $P'$, then $\{A,B\}$ is tight.
\end{LEM}

\subsection{Tree-decompositions}

A \emph{tree-decomposition} of a graph $G$ is a pair $(T, \cV)$ of a tree $T$ together with a family $\cV = (V_t)_{t \in V(t)}$ of subsets of $V(G)$ such that $\bigcup_{t \in T} G[V_t] = G$, and such that for every vertex~$v \in V(G)$, the graph $T[\{t\in T : v\in V_t\}]$ is connected. We call the sets~$V_t \in \cV$ the \emph{bags} and their induced subgraphs $G[V_t]$ the \emph{parts} of this \td. The \emph{adhesion} of $(T, \cV)$ is the maximal size of a set $V_t \cap V_{t'}$ for edges $\{t, t'\} \in E(T)$.

It is well-known (see e.g.\ \cite{DiestelBook16noEE}*{Ch.\ 12.3} for a proof) that if $e = \{t_1, t_2\}$ is an edge of $T$, then $V_{t_1} \cap V_{t_2}$ separates $U_1 := \bigcup_{s \in V(T_1)} V_s$ and $U_2 := \bigcup_{s \in V(T_2)} V_s$ where $T_1 \ni t_1$ and~$T_{2} \ni t_2$ are the two components of $T-e$. We say that $e$ \emph{induces} the separation $s_{e} := \{U_1, U_2\}$ of $G$. Let further $\vs_{(t_1, t_2)} := (U_1, U_2)$ be the separation of $G$ \emph{induced} by the oriented edge $(t_1, t_2)$.

It is easy to check that the separations induced by the (inwards oriented) edges incident with a node $t \in T$ form a star. We call this star $\sigma_t := \{\vs_{(u,t)} : (u, t) \in \vE(T)\}$ the \emph{star associated with the node $t$}.

A \td\ $(T, \cV)$ is \emph{tight} if every separation induced by an edge of $T$ is tight. $(T, \cV)$ \emph{(efficiently) distinguishes} two profiles if there is an edge $e \in E(T)$ such that $s_{e}$ distinguishes them (efficiently).

A \td\ $(T, \cV)$ of $G$ is \emph{canonical} if the construction $\Psi$ of $(T, \cV)$ commutes with all isomorphisms $\phi: G \rightarrow G'$: if $\phi$ maps the bags $V_t$ of $(T, \cV)$ to bags $V'_{t'}$ of $(T', \cV') := \Psi(G')$ such that $t \mapsto t'$ is an isomorphism $T \rightarrow T'$.

If~$(T, \cV)$ and $(\tilde{T}, \tilde{\cV})$ are both \tds\ of~$G$, then $(T, \cV)$ \emph{refines} $(\tilde{T}, \tilde{\cV})$ if the set of separations induced by the edges of~$T$ is a superset of the set of separations induced by the edges of~$\tilde{T}$.

\subsection{Blocks}

For some $k \in \N$, a \emph{$k$-block} in a graph $G$ is a maximal set $b$ of at least $k$ vertices such that no two vertices $v, w \in b$ can be separated in $G$ by removing fewer than $k$ vertices other than $v,w$. A set $b \subseteq V(G)$ is a \emph{block} if it is a $k$-block for some $k \in \N$.

It is straightforward to check that every $k$-block in $G$ \emph{induces} a regular $k$-profile in $G$ by orienting $\{A,B\} \in S_k(G)$ as $(A,B)$ if and only if~$b \subseteq B$. Moreover, distinct $k$-blocks induce distinct $k$-profiles.
We say that a \td\ of $G$ \emph{(efficiently) distinguishes} two blocks in $G$ if it (efficiently)  distinguishes their induced profiles.

A $k$-block $b$ in $G$ is \emph{separable} if it is the interior of some star in $S_k(G)$, i.e.\ if there exists a star $\sigma \subseteq \vS_k(G)$ such that $\interior(\sigma) = b$. 
We need the following equivalent characterization of separable $k$-blocks:

\begin{LEM}{\cite{CG14:isolatingblocks}*{Lemma 4.1}}\label{lem:CharacterizationOfSeparableBlocks}
    Let $b$ be a $k$-block in a graph $G$. Then $b$ is separable if and only if $\abs{N_G(C)} < k$ for all components~$C$ of $G-b$.
\end{LEM}

\section{Refining tree-decompositions}

In this section we prove \cref{thm:DisplayingBlocksRefining} and then derive \cref{thm:DisplayingkBlocksFinitek,thm:DisplayingkBlocksFiniteMax,thm:DisplayingkBlocksLocFink,thm:DisplayingkBlocksLocFinMax} from it. For this, we first show the following lemma. It asserts that given a part of a tight \td\ which contains a $k$-block, then we can further decompose that part in a star-like way so that the central bag of that decomposition is equal to the $k$-block:

\begin{LEM}\label{lem:StarWhoseInteriorIsABlock}
    Let $k \in \N$, and let $b$ be a separable $k$-block in a graph $G$. Further, let $\sigma \subseteq \vS_{\aleph_0}(G)$ be a star of tight separations such that $b \subseteq \interior(\sigma)$. Then there exists a star $\rho_b^\sigma \subseteq \vS_k(G)$ such that $\sigma \leq \rho_b^\sigma$ and $\interior(\rho_b^\sigma) = b$.
    Moreover, $\rho_b^\sigma$ can be chosen so that if $\phi: G \rightarrow G'$ is an isomorphism, then $\phi(\rho_b^\sigma) = \rho_{\phi(b)}^{\phi(\sigma)}$.
\end{LEM}

\begin{proof}
    Let $\cC' := \{C \in \cC(G - b) : V(C) \cap B \neq \emptyset \text{ for all } (A,B) \in \rho\}$ be the set of all components of $G-b$ that are not completely contained in the strict small side $G[A \setminus B]$ of some $(A,B) \in \sigma$. Further, for every component $C \in \cC'$, let $\sigma_C := \{(A,B) \in \sigma : A \cap V(C) \neq \emptyset\}$, and set
    \[
    (X_C, Y_C) := \bigg(V(C) \cup N_G(C) \cup \bigcup_{(A,B) \in \sigma_C} A \setminus B,\; V(G) \setminus \bigg(V(C) \cup \bigcup_{(A, B) \in \sigma_C} A \setminus B\bigg)\bigg)
    \] 
    (see \cref{fig:Sketch}). Note that $N_G(C) \subseteq b \subseteq B$ for all $(A,B) \in \sigma$, and thus $X_C \cap Y_C = N_G(C)$.
    
    \begin{figure}
        \centering
        \definecolor{dgrey}{rgb}{0.3,0.3,0.3}
\definecolor{dblue}{rgb}{0,0,0.7}
\definecolor{dgreen}{rgb}{0,0.7,0}
\definecolor{amy}{rgb}{0.6,0.4,0.8}
\scalebox{0.55}{%
%\begin{tikzpicture}
\begin{tikzpicture}[auto,rotate=270,transform shape]
%Block
\draw [line width=1.4pt] plot [smooth cycle, tension=0.3] coordinates {(7,11) (9.4,10.9) (9.5,7.5) (10,6) (9.9,5) (6,5.5) (4.6,3.7) (4,4) (3,5) (3,6) (5,8) (5,10)};
%Separation top left
\draw [dgreen, line width=1.3pt] plot [smooth, tension=0.4] coordinates {(0.75,10.5) (1.25,11.5) (2.5,12.5) (3.5,13) (4.75,13.25)};
\draw [dgreen, line width=1.3pt] plot [smooth, tension=0.4] coordinates {(0.5,10.75) (1.5,11) (2.5,11.5) (3.75,12.5) (4.5,13.5)};

%Separation top right
\draw [dgreen, dashed, line width=1.3pt] plot [smooth, tension=0.4] coordinates {(6.7,13.1) (8,12.5) (9.2, 11.5) (10,10.25) (10.3,9.2)};
\draw [dgreen, dashed, line width=1.3pt] plot [smooth, tension=0.4] coordinates {(6.9,13.4) (7.6,12) (8.5,11) (9.6,10) (10.5,9.4)};
%Separation bottom right
\draw [dgreen, dashed, line width=1.3pt] plot [smooth, tension=0.4] coordinates {(7.4,3.2) (8.2,4.5) (9,5.5) (10.5,6.5) (11.75,7.1)};
\draw [dgreen, dashed, line width=1.3pt] plot [smooth, tension=0.4] coordinates {(7.2,3.5) (9,4.2) (10,5) (10.5,5.6) (11.6,7.4)};
%Component \tilde{C}_A
\draw [line width=1.3pt] plot [smooth, tension=0.7] coordinates {(0,5.75) (1,5.85) (1.75,5.5) (2.25,5) (2.9,4) (3.15,3) (2.6,1.5) (1.5,0.5)};
\draw [line width=1.3pt] (1.3,5.79) -- (1.5,6.5);
\draw [line width=1.3pt] (1.7,5.52) -- (1.9,6);
%\draw [line width=1.3pt] (2.25,5) -- (2.5,5.25);
\draw [line width=1.3pt] (2.17,5.12) -- (2.45,5.4);
\draw [line width=1.3pt] (2.58,4.6) -- (2.9,4.8);
%\draw [line width=1.3pt] (2.8,4.15) -- (3.25,4.25);
\draw [line width=1.3pt] (2.8,4.15) -- (3.2,4.35);
%\draw [line width=1.3pt] (3.1,3.6) -- (3.75,3.75);
\draw [line width=1.3pt] (3.1,3.6) -- (3.8,3.8);
%Non-tight Component in A
\draw [line width=1.3pt] plot [smooth, tension=0.6] coordinates {(3,1) (3.5,2) (4.1,2.4) (4.6,2) (4.75,1)};
%\draw [line width=1.3pt] (3.95,2.35) -- (3.8,3.4);
\draw [line width=1.3pt] (3.8,2.3) -- (3.8,3.4);
\draw [line width=1.3pt] (4.1,2.4) -- (4.25,3.3);
\draw [line width=1.3pt] (4.35,2.3) -- (4.7,3.2);
%Component C
\draw [blue, line width=1.5pt] plot [smooth, tension=0.5] coordinates {(3.85,14) (4.45,13) (4.55,12) (4,10.75) (3.25,9.5) (2.75,8) (2.5,6.5) (2.8,5) (3.45,4) (3.5,2.5) (2,0)}; %{(3.85,14) (4.45,13) (4.55,12) (4,10.75) (3.25,9.5) (2.75,8) (2.6,6.5) (3,5) (3.5,4) (3.5,2.5) (2,0)};
%\draw [blue, line width=1.3pt] (4.55,12) -- (5.3,10.8);
\draw [blue, line width=1.3pt] (4.55,12) -- (5.5,11);
\draw [blue, line width=1.3pt] (4,10.75) -- (4.8,10.25);
\draw [blue, line width=1.3pt] (3.25,9.5) -- (4.5,9);
%\draw [blue, line width=1.3pt] (2.75,8) -- (4,7.5);
\draw [blue, line width=1.3pt] (2.75,8) -- (4.15,7.75);
%\draw [blue, line width=1.3pt] (2.6,6.5) -- (3.5,6.5);
\draw [blue, line width=1.3pt] (2.6,6.5) -- (3.3,6.5);
%Neighbourhood of C
\draw [dblue, line width=1.5pt, name path=two] plot [smooth, tension=0.5] coordinates {(6,10.63) (5.7,9) (5.5,7.5) (4.4,6) (4,5) (4.8,3.5)};

%Fill N(C)
\draw [line width=0.1pt, name path=one] plot [smooth, tension=0.3] coordinates {(6,10.63) (5,10) (5,8) (3,6) (3,5) (4,4) (4.8,3.5)};
\tikzfillbetween[of=one and two, on layer=ft] {pattern color=amy, pattern=north west lines};
%\tikzfillbetween[of=one and two] {color=lblue};

%X_C
\draw [purple, line width=1.5pt] plot [smooth, tension=0.5] coordinates {(4.5,14) (5,13.4) (5.65,12) (6,10.63) (5.7,9) (5.5,7.5) (4.4,6) (4,5) (4.8,3.5) (5,3) (5.2,1.5) (5,0)};
%Y_C
\draw [purple, line width=1.5pt] plot [smooth, tension=0.5] coordinates {(6.5,14) (6,12.5) (6,10.63) (5,10) (5,8) (3,6) (3,5) (4,4) (4.8,3.5) (5.7,2) (7,0.5)};

%Separation bottom left
\draw [dgreen, line width=1.5pt] plot [smooth, tension=0.4] coordinates {(1.1,7.6) (2.25,7) (3.25,6) (4,5) (5,3)};
\draw [dgreen, line width=1.5pt] plot [smooth, tension=0.4] coordinates {(1.4,7.9) (2,6.5) (3,5) (4,4) (5.2,3.2)};

%Component C'
\draw [blue, line width=1.5pt] plot [smooth, tension=0.6] coordinates {(10.4,3) (10,4) (10.1,4.65) (11,4.75) (11.75,4.2)};
\draw [blue, line width=1.3pt] (9.95,4.25) -- (9.5,4.4);
\draw [blue, line width=1.3pt] (10.1,4.65) -- (9.95,4.75);
\draw [blue, line width=1.3pt] (10.4,4.75) -- (10.35,5.2);
\draw (11.4,-0.4) node[rotate=90, blue, anchor=south west] {\huge $C' \in \cC(G-b)\setminus \cC'$};

%Names
\draw (1.5,3) node[rotate=90, anchor=south west] {\huge $\tilde{C}_A$};
\draw (1.5,8.5) node[rotate=90, color=blue, anchor=north west] {\huge $C \in \cC'$};
\draw (6.8,6) node[rotate=90, color=amy, anchor=south west] {\huge $N(C) = X_C \cap Y_C$};
\draw (8.5,8.5) node[rotate=90, anchor=south west] {\huge $b$};
\draw (4.7,0.5) node[rotate=90, color=purple, anchor=south east] {\huge $X_C$};
\draw (7.5,1.9) node[rotate=90, color=purple, anchor=south east] {\huge $Y_C$};
\draw (11.5,8) node[rotate=90, color=dgreen, anchor=south west] {\huge $\sigma$};
\draw (0.9,9) node[rotate=90, color=dgreen, anchor=south east] {\LARGE $(A,B) \in \sigma_C$};
%\draw (6.7,12.2) node[rotate=90, color=dgreen, anchor=south west] {\LARGE $(A',B') \not\in \sigma_C$};
\draw (11.1,9.9) node[rotate=90, color=dgreen, anchor=south west] {\LARGE $(A',B') \not\in \sigma_C$};
%\end{tikzpicture}
%Old Names
%\draw (1,2.5) node[color=dgrey, anchor=south west] {\huge $\tilde{C}_A$};
%\draw (1,9.5) node[color=blue, anchor=south west] {\huge $C \in \cC'$};
%\draw (5.5,7) node[color=amy, anchor=south west] {\huge $N(C)$};
%\draw (7.5,9) node[anchor=south west] {\huge $b$};
%\draw (5,0) node[color=dblue, anchor=south east] {\huge $X_C$};
%\draw (10.5,8) node[color=dgreen, anchor=south west] {\huge $\sigma$};
%\draw (2.7,8) node[color=dgreen, anchor=south east] {\LARGE $(A,B) \in \sigma_C$};
%\draw (7.8,12.7) node[color=dgreen, anchor=south west] {\LARGE $(A',B') \not\in \sigma_C$};
\end{tikzpicture}
}%
        %\vspace{-5mm}
        \caption{A component $C \in \cC'$ and the arising separation $(X_C,Y_C)$. The separations $(A,B) \in \sigma_C$ are indicated with solid lines, the separations $(A', B') \in \sigma \setminus \sigma_C$ are indicated with dashed lines. The component $\tilde{C}_A$ of $G - (A \cap B)$ is contained in $C$.}
        \label{fig:Sketch}
    \end{figure}
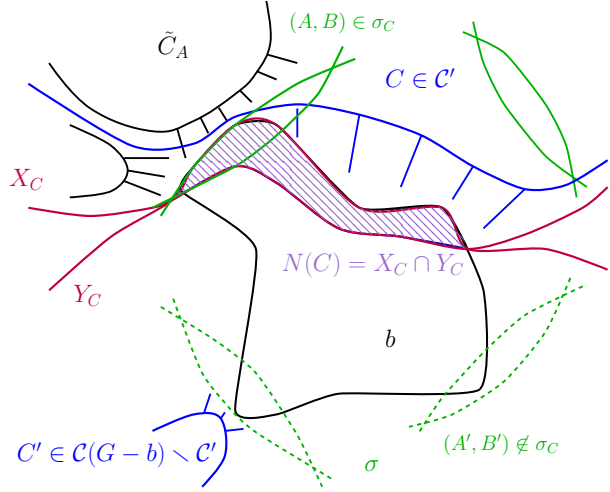
     
    Let us first show that $\{X_C, Y_C\}$ is a separation of $G$. Clearly, its sides cover $V(G)$, so it remains to prove that $N_G(X_C\setminus Y_C) \subseteq X_C$. 
    By the definition of $\{X_C, Y_C\}$, this is the case if
    $N_G(A\setminus B) \subseteq V(C) \cup N_G(C)$ for all $(A,B) \in \sigma_C$.
    So let $(A,B) \in \sigma_C$ be given. Then $V(C) \cap A \neq \emptyset$, and moreover $V(C) \cap B \neq \emptyset$ because $C \in \cC'$, which implies that $V(C) \cap (A \cap B) \neq \emptyset$ as~$C$ is connected. Since $\{A,B\}$ is tight, there is a component $\tilde{C}_A \subseteq G[A\setminus B]$ of $G - (A \cap B)$ such that $N_G(\tilde{C}_A) = A \cap B$. 
    In particular, there is an edge between~$\tilde{C}_A$ and~$C$. As~$C$ is a component of $G - b$ and $\tilde{C}_A$ is connected and disjoint from~$b$ by assumption, this implies that $\tilde{C}_A \subseteq C$.
    Hence, $N_G(A \setminus B) = A \cap B = N_G(\tilde{C}_A) \subseteq N_G(C) \cup V(C)$, and thus $\{X_C,Y_C\}$ is a separation of~$G$. 
    In particular,~$\{X_C,Y_C\}$ has order $\abs{N_G(C)}$. 
    
    Since $\abs{N_G(C)} < k$ by \cref{lem:CharacterizationOfSeparableBlocks}, this implies that $\{X_C, Y_C\} \in S_k(G)$.
    Now set
    \[
    \rho_b^\sigma := \{(X_C,Y_C) : C \in \cC'\} \cup \{(A,B) \in \sigma : A \cap B \subseteq b\}.
    \]
    Since every $(A,B) \in \sigma$ is tight, \cref{lem:CharacterizationOfSeparableBlocks} implies that every $(A,B) \in \sigma$ with $A \cap B \subseteq b$ has order less than~$k$.
    As also every $\{X_C,Y_C\}$ for $C \in \cC'$ is of order less than $k$ by the argument above, it follows that $\rho_b^\sigma \subseteq \vS_k(G)$. 
    We claim that $\rho_b^\sigma$ is as desired.
    
    First, we prove that $\rho_b^\sigma$ is a star. For this, we show that $(X_C,Y_C) \leq (Y_{C'}, X_{C'})$ for distinct $C \neq C' \in \cC'$ and that $(X_C, Y_C) \leq (B,A)$ for all $C \in \cC'$ and $(A,B) \in \sigma$ with $A \cap B \subseteq b$. Since $\sigma$ is a star itself, this then concludes the proof that $\rho_b^\sigma$ is a star.
    
    We first show the former. To this end, let two distinct components $C \neq C' \in \cC'$ be given. Then they cannot both meet the small side~$A$ of the same separation $(A, B) \in \sigma$, as otherwise $\tilde{C}_A \subseteq C \cap C'$ by the argument above, and then $C = C'$.
    Therefore, $\sigma_C \cap \sigma_{C'} = \emptyset$, and thus $X_C \setminus Y_C = V(C) \cup \bigcup_{(A,B) \in \sigma_C} A \setminus B$ and $X_{C'} \setminus Y_{C'} = V(C') \cup \bigcup_{(A,B) \in \sigma_{C'}} A \setminus B$ are disjoint. %Note that $N_G(X_C \setminus Y_C) = N_G(C) \subseteq b \subseteq Y_{C'}$ and hence $X_C \subseteq Y_{C'}
    Hence $(X_C, Y_C) \leq (Y_{C'}, X_{C'})$. 
    
    Now let $C \in \cC'$ and $(A,B) \in \sigma$ with $A \cap B \subseteq b$ be given. Then $V(C) \cap B \neq \emptyset$ by the definition of $\cC'$, which implies that $V(C) \subseteq B \setminus A$ as $C$ is connected and avoids $b \supseteq A \cap B$. Thus, $(X_C, Y_C) \leq (B,A)$. 
    
    Second, we show that $\sigma \leq \rho_b^\sigma$. For this, let $(A,B) \in \sigma$ be given. We need to find a separation $(A', B') \in \rho_b^\sigma$ with $(A,B) \leq (A', B')$. If $A \cap B \subseteq b$, then $(A,B) \in \rho_b^\sigma$ is as desired. Otherwise, $A \cap B$ meets a component~$C$ of $G - b$, which then has to lie in $\cC'$. In particular, $(A,B) \in \sigma_C$ and $(A,B) \leq (X_C, Y_C)$ by the definition of $\{X_C, Y_C\}$.
    Since $(X_C, Y_C) \in \rho_b^\sigma$ as $C \in \cC'$, this completes the proof that $\sigma \leq \rho_b^\sigma$.

    Next, we show that $\interior(\rho_b^\sigma) = b$. For this, we first observe that $b$ is disjoint from every component~$C$ of~$G - b$ and from every strict small side $A \setminus B$ of every $(A,B) \in \sigma$ by assumption. By the definition of~$\rho_b^\sigma$, this implies that $b \subseteq \interior(\rho_b^\sigma)$. 
    Moreover, every vertex $v \in V(G) \setminus b$ is contained in a component~$C$ of $G - b$. If $C \in \cC'$, then $v$ lies in the strict small side $X_C \setminus Y_C$ of $(X_C, Y_C)$ by definition, and hence $v \notin \interior(\rho_b^\sigma)$. Otherwise, there is a separation $(A,B) \in \sigma$ such that $v \in V(C) \subseteq A \setminus B$. Since $\sigma \leq \rho_b^\sigma$ as show earlier, this implies that $v \notin \interior(\rho_b^\sigma)$.
    Therefore, $\interior(\rho_b^\sigma) \subseteq b$.

    We are left to show the `moreover'-part. For this, let $\phi: G \rightarrow G'$ be an isomorphism. We show that $(\phi(X_C), \phi(Y_C)) = (X_{\phi(C)}, Y_{\phi(C)})$, which clearly implies the assertion. For this, note that $\phi(b)$ is a $k$-block in $G'$, that $\phi(C)$ is a component of $G'-\phi(b)$, and that $\phi(\sigma)$ is a star of tight separations in $\vS_k(G')$ with $\phi(b) \subseteq \interior(\phi(\sigma))$. Thus, $\{X_{\phi(C)}, Y_{\phi(C)}\}$ is defined. Moreover, if $V(C) \cap A \neq \emptyset$ for some $(A,B) \in \sigma$, then $\phi(V(C)) \cap \phi(A) \neq \emptyset$. Hence, $\phi(\sigma)_{\phi(C)} = \phi(\sigma_C)$, which, by the definition of $\{X_C, Y_C\}$ and because $\phi$ is an isomorphism, implies that $(\phi(X_C), \phi(Y_C)) = (X_{\phi(C)}, Y_{\phi(C)})$.
\end{proof}

We can now prove \cref{thm:DisplayingBlocksRefining}.

\begin{proof}[Proof of \cref{thm:DisplayingBlocksRefining}]
    Applying \cref{lem:StarWhoseInteriorIsABlock} to every star $\sigma_t$ that is associated with a node $t \in \tilde{T}$ such that~$\tilde{V}_t$ contains a block $b$ in $\cB$ yields stars $\rho_b^{\sigma_t}$ with $\sigma_t \leq \rho_b^{\sigma_t}$ and $\interior(\rho_b^{\sigma_t}) = b$. 

    We now construct the desired \td\ $(T, \cV)$. For this, we first define \tds\ $(T^t, \cV^t)$ of the parts $G[\tilde{V}_t]$ as follows. If $\tilde{V}_t$ does not contain a block from $\cB$, then we set $T^t := (\{t\}, \emptyset)$ and $V^t_t := \tilde{V}_t$. Otherwise, if $b$ is the (unique) block from $\cB$ that is contained in $\tilde{V}_t$, then we let $T^t$ be the star with centre $t$ and with $|\rho_b^{\sigma_t}|$ many leaves $u_{(A,B)}$, one for each $(A,B) \in \rho_b^{\sigma_t}$. 
    Further, we set $V^t_t := b = \interior(\rho_b^{\sigma_t})$ and $V^t_{u_{(A,B)}} := A \cap \tilde{V}_t$ for all $(A,B) \in \rho_b^{\sigma_t}$. It is straightforward to check that $(T^t, \cV^t)$ is a \td\ of~$G[V_t]$. 
    
    We then let $T$ be the tree obtained from the disjoint union over the trees $T^t$ by adding for every edge $\{t_1,t_2\} \in \tilde{T}$ the edge $\{v_1, v_2\}$ where $v_i = t_i$ if $T^{t_i} = (\{t_i\}, \emptyset)$ and $v_i := u_{(A,B)}$ where $(A, B) \in \rho^{\sigma_{t_i}}_b$ is the unique separation with $\vs_{(t_{3-i}, t_i)} = (U_{3-i}, U_i) \leq (A,B)$ otherwise. Note that such a separation exists because $(U_{3-i}, U_i) \in \sigma_{t_i} \leq \rho_b^{\sigma_{t_i}}$. Further, we set $V_s := V^t_s$ for all $s \in T$ where $t$ is the unique node of~$\tilde{T}$ such that $s \in T^t$. It is straightforward to check that $(T, \cV)$ is a \td\ of $G$. Moreover, by construction, for every edge $\{t_1, t_2\} \in \tilde{T}$, the edge $\{v_1, v_2\}$ of $T$ induces the same separation of $G$, i.e.\ $s_{(t_1, t_2)} = s_{(v_1, v_2)}$, so $(T, \cV)$ refines~$(\tilde{T}, \tilde{\cV})$. 
    Finally, by the `moreover'-part of \cref{lem:StarWhoseInteriorIsABlock}, $(T, \cV)$ is canonical if $(\tilde{T}, \tilde{\cV})$ is canonical.
    Hence, $(T, \cV)$ is as desired.
\end{proof}

\begin{proof}[Proof of \cref{thm:DisplayingkBlocksLocFink}]
    By \cites{CanonicalTreesofTDs,InfiniteSplinters}, $G$ admits a canonical \td\ $(\tilde{T}, \tilde{\cV})$ that efficiently distinguishes all the $k$-profiles in $G$.\footnote{In \cite{CanonicalTreesofTDs}*{Theorem 7.3} Carmesin, Hamann and Miraftab prove that $G$ admits a canonical \td\ that distinguishes efficiently all its \emph{robust} $k$-profiles. In \cite{InfiniteSplinters}*{Theorem 6.6} Elbracht, Kneip and Teegen give an independent proof of this result. Here, it can be seen from the proof that one may omit `robust', as the authors remark in the preliminary section.} In particular, $(\tilde{T}, \tilde{\cV})$ is tight by \cref{lem:EffDistinguisherAreTight}. Moreover, since every $k$-block induces a $k$-profile, and since distinct $k$-blocks induce distinct $k$-profiles, $(\tilde{T}, \tilde{\cV})$ distinguishes all $k$-blocks in $G$. Apply \cref{thm:DisplayingBlocksRefining} to $(\tilde{T}, \tilde{\cV})$ and the set~$\cB$ of all separable $k$-blocks in $G$.
\end{proof}

\begin{proof}[Proof of \cref{thm:DisplayingkBlocksLocFinMax}]
    By \cite{InfiniteSplinters}*{Theorem 6.6} and \cite{PRToTsInLocFinGraphs}*{Theorem~5.4}\footnote{See also \cite{PRToTsInLocFinGraphs}*{Theorem 1 and the comment after the proof of Theorem 1}.}, $G$ admits a canonical \td\ $(\tilde{T}, \tilde{\cV})$ that efficiently distinguishes all its maximal robust profiles. In particular, $(\tilde{T}, \tilde{\cV})$ is tight by \cref{lem:EffDistinguisherAreTight}. Apply \cref{thm:DisplayingBlocksRefining} to~$(\tilde{T}, \tilde{\cV})$ and the set $\cB$ of all separable blocks in $G$ that induce a maximal robust profile.
\end{proof}

\begin{proof}[Proof of \cref{thm:DisplayingkBlocksFinitek}]
    Apply \cref{thm:DisplayingkBlocksLocFink}.
\end{proof}

\begin{proof}[Proof of \cref{thm:DisplayingkBlocksFiniteMax}]
    Apply \cref{thm:DisplayingkBlocksLocFinMax}.
\end{proof}

\bibliographystyle{amsplain}
\bibliography{local.bib}

\end{document}